\documentclass[10pt]{amsart} 
\usepackage{amssymb,latexsym,amsthm,enumerate,mathtools,amsmath,microtype} 
\usepackage{xcolor}
\usepackage[backref, pdfencoding=auto,  psdextra]{hyperref}
\usepackage[abbrev, nobysame]{amsrefs}

\renewcommand{\MR}[1]{} 
\renewcommand{\PrintDOI}[1]{}
\BibSpec{arxiv}{
    +{}  {\PrintAuthors}                {author}
    +{,} { \textit}                     {title}
    +{:} { \textit}                     {subtitle}
    +{,} { \PrintContributions}         {contribution}
    +{.} { \PrintPartials}              {partial}
    +{,} { }                            {journal}
    +{}  { \PrintDatePV}                {date}
     +{,} { \eprint}        {url}
}

\usepackage{tikz-cd}
\usetikzlibrary{babel, shapes.geometric}
\usetikzlibrary {arrows.meta,positioning}

\makeatletter
\newcommand{\DeclareMathActive}[2]{
  
  \expandafter\edef\csname keep@#1@code\endcsname{\mathchar\the\mathcode`#1 }
  \begingroup\lccode`~=`#1\relax
  \lowercase{\endgroup\def~}{#2}
  \AtBeginDocument{\mathcode`#1="8000 }
}

\newcommand{\std}[1]{\csname keep@#1@code\endcsname}
\iftutex
  \patchcmd{\newmcodes@}{\Umathcodenum `\-\relax}{\std@minuscode\relax}{}{\ddt}
\else
 
\fi
\AtBeginDocument{\edef\std@minuscode{\the\mathcode`-}}
\def\xyz{\@ifstar{\dast}{\std{*}}}
\makeatother
\def\dast{\mathbin{\mathpalette{}{{\std{*}\std{*}} }}}
\DeclareMathActive{*}{\xyz}

\newtheoremstyle{dotless}{}{}{\itshape}{}{\bfseries}{}{ }{}

\newtheorem{Theorem}{Theorem}[section]
\newtheorem{Prop}[Theorem]{Proposition} 

\newtheorem{corollary}[Theorem]{Corollary} 
 
\newtheorem{lemma}[Theorem]{Lemma}

\theoremstyle{definition}

\newtheorem*{Acknowledgements}{Acknowledgements}

\newcommand{\ub}{\overline{\beta}}
\newcommand{\lb}{\underline{\beta}}

\DeclareMathOperator{\Ore}{Ore}

\DeclarePairedDelimiter\abs{\lvert}{\rvert}

\newcommand{\cG}{\mathcal {G}} 

\newcommand{\F}{{\mathbb F}} 
\newcommand{\nn}{{\mathbb N}} 
\newcommand{\zz}{{\mathbb Z}} 
\newcommand{\rr}{{\mathbb R}} 
\renewcommand{\qq}{{\mathbb Q}} 

\newcommand{\supp}{\operatorname{supp}}

\newcommand{\osum}{\oplus} 
\newcommand{\gf}{\varphi} 
\newcommand{\gm}{\mu} 
\newcommand{\gt}{\tau} 

\newcommand{\FG}{\F G}
\newcommand{\FH}{\F H}
\newcommand{\FK}{\F K}
\newcommand{\FJ}{\F J}
\newcommand{\DG}{D_{\FG}}

\DeclareMathOperator{\im}{Im}
\renewcommand{\ker}{\operatorname{Ker}}
\DeclareMathOperator{\Aut}{Aut}
\DeclareMathOperator{\Div}{Div}

\newcommand{\nsgp}{\mathrel{\triangleleft}} 
\newcommand{\onto}{\twoheadrightarrow}

\numberwithin{equation}{section} 
\title[Orders and Fibering]
{Orders and Fibering}

\begin{document}\begin{abstract}
    In 2018, Kielak gave a virtual fibering criterion for RFRS groups.
    In this paper, we present a simpler proof of this.
\end{abstract}
\author{Boris Okun} \address{Department of Mathematical Sciences, University of Wisconsin--Milwaukee, Milwaukee, WI 53201}
\email{okun@uwm.edu}

\author{Kevin Schreve} \address{Department of Mathematics, Louisiana State University, Baton Rouge, LA~70806}
\email{kschreve@lsu.edu}

\date{\today}

\maketitle

\section{Introduction} It is typically difficult to compute the $L^2$-homology groups of a space.
One useful tool is L\"{u}ck's Mapping Torus Theorem \cite{l94}: if $X$ is a finite CW-complex and $f: X \to X$ a self map, then the $L^2$-homology groups of $\widetilde T_f$, the universal cover of the mapping torus of $f$, vanish in all dimensions.

Remarkably, through work of Kielak in \cite{k20a}, and extended by Jaikin-Zapirain in \cite{j21} and Fisher in \cite{f24}, there is now essentially a converse to L\"{u}ck's theorem for aspherical complexes with residually finite rationally solvable (RFRS) fundamental group.
The main point of this paper is to provide a different proof of this.
The RFRS condition is somewhat technical, but it turns out that the class of RFRS groups is the same as the class of residually (virtually abelian and locally indicable) groups, and the latter description is more relevant for this paper\footnote{In the first version of this article, we claimed that residually (virtually abelian and locally indicable) groups formed a larger class than RFRS groups.
Sam Fisher and Dawid Kielak explained to us that they are the same.
We present their argument in Section \ref{s:resid}.}
. The correct statement of the converse requires some setup.
First, the proof of the mapping torus theorem is robust, and in particular works for analogues of $L^2$-homology groups over an arbitrary field $\F$.
For a RFRS group $G$, these are defined as the homology groups $H_\ast(G; D_{\FG})$, where the coefficients are a certain skew field $D_{\FG}$ containing the group ring $\FG$.
For $\F = \qq$, the existence of $D_{\qq G}$ follows from the strong Atiyah conjecture, which is known for RFRS groups by work of Schick \cites{s00,s02}, and Linnell \cite{l93} showed that the usual $L^2$-Betti numbers of $G$ are equal to $b_\ast(G; D_{\qq G})$.
Jaikin-Zapirain \cite{j21} showed that $D_{\FG}$ exists for RFRS groups and all fields, and a theorem of Hughes \cite{h70} implies that $D_{\FG}$ is unique up to $\FG$-module isomorphism; this implies a number of useful properties that we review in Section \ref{s:linnell}.
Secondly, it turns out that vanishing of $H_\ast(G; D_{\FG})$ is a commensurability invariant, so vanishing should only imply a mapping torus structure for a finite index subgroup $H$ of $G$.
Finally, the correct analogue of $X$ being a finite aspherical complex is the existence of a character $\gf: H \to \zz$ with $\ker \gf$ of type $FP_n(\F)$.
\begin{Theorem}[\cites{k20a,f24}]\label{t:main}
    Suppose $G$ is a virtually RFRS group that is type $FP_n(\F)$, and suppose $H_i(G; D_{\FG}) = 0$ for $i \le n$.
    Then there is a finite index subgroup $H$ and a map $\gf: H \to \zz$ such that $\ker \gf$ is of type $FP_n(\F)$.
\end{Theorem}

This theorem implies Agol's fibering criterion for $3$-manifolds \cite{a08}: if $M$ is a closed aspherical $3$-manifold with virtually RFRS fundamental group, then $M$ has a finite cover which fibers over $S^1$.
Indeed, in this case, work of Dodziuk~\cite{d79} and Lott--L\"{u}ck~\cite{ll95} on the $L^2$-homology of $3$-manifolds, combined with the aforementioned work of Linnell and Schick, imply that $H_1(\pi_1M; D_{\qq \pi_1M}) = 0$, so Theorem \ref{t:main} guarantees a finite index subgroup mapping to $\zz$ with $FP_1(\qq)$ kernel.
Type $FP_1(\qq)$ is equivalent to being finitely generated, and a theorem of Stallings \cite{s61} implies that the corresponding finite cover of $M$ is fibered.

The proof of Theorem \ref{t:main} has three major ingredients.
First, the honest (i.e.
non-virtual) fiberings of $G$ are well understood due to the work of Bieri--Neumann--Strebel--Renz, who defined certain conical subsets $\Sigma^n(G; \zz) \subset H^{1}(G; \rr)$ which control finiteness properties of kernels of homomorphisms from $G$ onto abelian groups.
In particular, a special case of a theorem of Bieri and Renz~\cite{br88}*{Theorem 5.1} shows that for a character $\gf: G \to \zz$, $\ker \gf$ is type $FP_n(\zz)$ if and only if $\pm \gf \in \Sigma^n(G; \zz)$, see also \cite{fgs10}*{Section 7}.

The second ingredient is a connection between BNSR invariants and homology with coefficients in \emph{Novikov rings}, found by Sikorav in his thesis \cite{s87}.
Given a character $\gf: G \to \rr$, the Novikov ring $\FG^\gf$ is defined as
\[
\FG^\gf := \{f: G \to \F | \supp(f) \cap \gf^{-1}(-\infty, t] \text{ is finite for all } t\}.
\]
with multiplication induced from $\F G$.
Sikorav proved that $\gf \in \Sigma^1(G, \zz)$ if and only if $H_1(G; \zz G^{\gf}) = 0$ and Schweitzer \cite{b07a}*{Appendix A.I} extended this to show that $\gf \in \Sigma^n(G, \zz)$ if and only if $H_i(G; \zz G^{\gf}) = 0$ for $i \le n$.
Fisher \cite{f24}*{Theorem 5.3} observed that the ring $\zz$ could be replaced with any ring $R$.

The last ingredient, which is the main discovery of Kielak, is the connection for RFRS groups between vanishing of $H_i(G; D_{\qq G})$ and vanishing of $H_{i}(H; \qq H^{\gf})$ for certain finite index subgroups $H$ and characters $\gf: H \to \zz$.
Roughly speaking, Kielak shows that $D_{\qq G}$ is covered by Novikov rings of finite index subgroups $\qq H^\gf$.
This argument was extended, in somewhat different terms, to other coefficients by Jaikin-Zapirain in the Appendix of \cite{j21}.

Theorem \ref{t:main} follows easily; the vanishing of $H_i(G; D_{\FG})$ for $i \le n$ lets one put the chain complex $C_\ast(G; D_{\FG})$ into a simple Smith normal form up to degree $n$.
If the elements of the change of basis matrices could be thought of as elements of $\FH^\gf$ for an appropriate $\gf: H \to \zz$, then the same matrices would bring $C_\ast(H; \FH^\gf)$ into the same Smith normal form, and we could conclude vanishing of the Novikov homology.

The precise covering argument is rather technical, since one has to compare characters coming from different finite index subgroups, and characters on $H$ do not extend to characters on $G$.

In this paper, we provide an alternative way to compare $D_{\FG}$ and $\FH^\gf$, which we find simpler.
Combining this with the above work relating vanishing of Novikov homology and finiteness properties gives another proof of Theorem \ref{t:main}.
We emphasize that these ideas came from trying to understand Kielak's proof, and the basic tools are the same.

The main tools are \emph{Malcev--Neumann crossed products}.
If $G$ is an ordered group, then Malcev and Neumann independently showed \cites{m48,n49} that the set of functions $f: G \to \F$ with well-ordered support endowed with convolution multiplication becomes a skew field $\F*_{<}G$ containing $\FG$.
In this case, the skew field $D_{\FG}$ is the smallest subfield of $\F*_{<}G$ containing the group ring $\FG$.

Our new idea is an alternative description of $D_{\FG}$ in the case that $G$ admits a normal series $\dots \nsgp K_i \nsgp K_{i-1} \nsgp \dots \nsgp K_1 \nsgp K_0 = G$ with $K_i/K_{i+1}$ orderable amenable and $\bigcap K_i = 1$.
If $G$ and each $G/K_i$ is orderable and the map $G \to G/K_i$ is order-preserving, our description is essentially the same as that given by Eizenbud and Lichtman in \cite{el87}*{Proposition 4.3}, see also \cite{l00}*{Theorem 5}: each element $f$ of $D_{\FG}$ eventually has finite intersection with the $K_i$-cosets, hence can be represented as a well-ordered series on $G/K_i$ with coefficients in $\FK_i$.

In our setting, even if $G$ and each $G/K_i$ are orderable, the subgroups $K_i$ are not necessarily convex, hence we have no order-preserving maps to the $G/K_i$.
Therefore, we work only with orders on the successive quotients $K_i/K_{i+1}$.
Our description is essentially that each element of $D_{\FG}$ can be represented iteratively as well-ordered series in the $K_i/K_{i+1}$, i.e.
each element decomposes as a well-ordered series on $K_0/K_1$ with coefficients which are well-ordered series in $K_1/K_2$ with coefficients which are \dots, eventually terminating in series with coefficients in $\FK_n$ for some $n$.

Though this description is a bit unwieldy, it becomes much more manageable for finitely generated residually (virtually abelian locally indicable) groups, especially if we are willing to pass to finite index subgroups.
In this case, $G$ admits a normal series with each $K_{i}/K_{i+1}$ finitely generated free abelian; in particular passing to finite index subgroups does not change the possible orders on successive quotients.
Even better, for each $n$ there is a finite index normal subgroup $H \nsgp G$ containing $K_{n}$ such that $H/K_{n}$ is free abelian.
If $f \in D_{\FH}$ this description of iterative series is the same as describing $f$ as a well-ordered series on $H/K_{n}$ with coefficients in $\FK_{n}$, with respect to the dictionary order induced from the subgroups $(H \cap K_{i})/(H \cap K_{i+1})$.
This dictionary order can be well-approximated by a character $\psi: H \to \zz$, and this allows us to think of elements in $D_{\FG}$ as matrices over $\FH^\psi$.
Therefore, we obtain a slightly more concrete description of the fibering characters than that given in \cite{k20a}; any sufficiently close approximation to these dictionary orders works.
\begin{Acknowledgements}
    We thank Dawid Kielak and Andrei Jaikin-Zapirain for comments on an earlier version of this article.
    We thank Grigori Avramidi for an illuminating conversation which resulted in Theorem~\ref{t:B}.
    We thank Sam Fisher and Dawid Kielak for providing an argument that residually (virtually abelian and locally indicable) groups are RFRS. We thank the anonymous referee for careful reading and useful comments.
    The second author was supported by the NSF grant DMS-2203325.
\end{Acknowledgements}

\section{Crossed products and Malcev--Neumann series}\label{s:cross}
Let $R$ be a ring with $1$, and $Q$ a group.
Let $c : R^\times \to \Aut(R)$ denote the left action by conjugation $c(r)(r')=r r' r^{-1}$.
Let $\gt : Q \to \Aut(R)$ and $\mu : Q \times Q \to R^\times$ be two functions satisfying
\begin{align}
    \gt(q) \circ \gt(q') &= c(\mu(q,q') ) \circ \gt(qq')\label{e:tau}
    \\
    \mu(q,q') \cdot \mu(qq',q'') &= \gt(q)( \mu(q',q'') ) \cdot \mu(q,q'q'')\label{e:mu}
\end{align}

The functions $\gt$ and $\mu$ are \emph{structure functions}, and turn $\osum_{Q} R$ into a \emph{crossed product ring} $R*Q$ by extending linearly the convolution product
\begin{equation}\label{e:conv}
    r q \cdot r' q' = r \gt(q)(r') \mu(q,q') qq'.
\end{equation}

Playing around with \ref{e:tau} and \ref{e:mu} yields the following lemma.
\begin{lemma}
    Suppose $R*Q$ is a crossed product with structure functions $\mu$ and $\tau$.
    \begin{enumerate}
        \item $\mu(1,1) = \mu(1,q)$.
        \item $\gt(1)=c(\mu(1,1))$.
        \item The identity of $R*Q$ is $\mu(1,1)^{-1}1$.
        \item If $u$ is a unit of $R$, then $uq$ is a unit of $R*Q$.
    \end{enumerate}
\end{lemma}

An equivalent definition is the following.
A ring $S$ is $Q$-\emph{graded} if $S=\osum_{q \in Q} R_{q}$ as abelian groups and the grading is compatible with multiplication: $R_{q}R_{q'} \subset R_{qq'}$.
Note that this implies that $R=R_{1}$ is a subring.
If in addition, each $R_{q}$ contains a unit of $S$, then $R_{q}$ is an $R$-module $R$-isomorphic to $R$ and $S$ is a crossed product of $R$ and $Q$, $S=R*Q$.

Clearly the structure maps definition gives a $Q$-grading, and the above lemma implies that each $R_{q}$ contains a unit.

In the other direction, pick units $s_{q} \in R_{q}$ and define:
\begin{align*}
    \gt(q) &= c(s_{q}) \\
    \mu(q,q') &= s_{q} s_{q'} s_{qq'} ^{-1}.
\end{align*}
We will always choose $s_{1}=1$ such that $1\cdot 1$ is 1 of $R*Q$.
Then $R$ embeds into $R*Q$ via the map $r \mapsto r 1$,

Suppose $K \nsgp G$ and $Q=G/K$.
The graded definition obviously applies to the decomposition of the group ring $RG$ over cosets of $K$ and gives an isomorphism $R G \cong RK*Q$.
Moreover, we can choose units $s_{q}$ to be given by a set-theoretic section $s$ of $G \to Q$.
The inverse isomorphism $R K*Q \cong R G$ is given by
\[
\sum_{q \in Q} x_{q} q \mapsto \sum_{q \in Q} x_{q} s_{q},
\]
where the coefficients $x_{q}\in R K$.
Note that the image of $\tau$ has only automorphisms induced by conjugation by elements of $G$, and the image of $\mu$ lies in $K \subset RK^\times$.
\begin{lemma}[\cite{p89}*{Lemma 1.3}]\label{l:decomp}
    Let $K \nsgp G \onto Q $.
    Given a crossed product $R*G$,
    \[
    R*G \cong (R*K)*Q
    \]
\end{lemma}
\begin{proof}
    Choose a section $g:Q \to G$.
    Then $G$ is a disjoint union of $K$-cosets and
    \[
    R*G=\bigoplus_{q \in Q} (R*K) g_{q}
    \]
    is the $Q$-grading.
    In other words, for a single term, $r (kg_{q}) = (r k) q$.
    In $R*G$ we have $(r\, kg_{q})(r\,'k'g_{q'})=r''\, kg_{q} k'g_{q'} $, and together with multiplication in $G$, $ kg_{q} k'g_{q'} = k'' g_{qq'}$, hence this $Q$-grading is compatible with multiplication in $R*G$.
\end{proof}

In terms of the structure maps $\gt,\gm$ for $R*G$, the structure maps for $(R*K)*Q$ are given by
\begin{align*}
    \hat\gt(q)(rk) &=\gt(g_{q})(r)\, g_{q} k g_{q}^{-1}, \\
    \hat\gm(q,q') &=\gm(g_{q}, g_{q'}) \gm(k, g_{qq'})^{-1}\, k ,\\
    \intertext{where} k &=g_{q} g_{q'} g_{qq'}^{-1} \in K.
\end{align*}

\subsection{Orderable groups} Recall that a group $G$ is \emph{orderable} if it admits a strict total ordering that is invariant under left and right multiplication.
A subgroup $C$ of an ordered group $G$ is \emph{convex} if $x,y \in C$ and $x < z < y$ implies $z \in C$.
If $C$ is convex and normal, there is an induced order on $G/C$, where $xC < yC$ if and only if $x< y$, hence the quotient map $G \to G/C$ is order-preserving.
Convex subgroups of $G$ are linearly ordered by inclusion, so a maximal proper convex subgroup $C$, if it exists, is unique, and hence normal.
In this case $G/C$ has no nontrivial proper convex subgroups.
This is equivalent to the induced order on $G/C$ being \emph{Archimedean}: for any $1 < x, y$ there is $n \in \nn$ such that $x < y^n$.
H\"{o}lder's theorem \cite{h01} implies that Archimedean orders are induced by injective characters to $\rr$.

If $G$ is generated by elements $1<g_{1}<\dots < g_{n}$, then a convex subgroup containing $g_{n}$ contains all generators, and hence is $G$.
Therefore, the union of convex subgroups not containing $g_n$ is the maximal proper convex subgroup, and we obtain the following:
\begin{lemma}\label{l:op}
    If $G$ is a finitely generated ordered group, then there is an order-preserving character $\gf: G \to \rr$ such that $\ker \gf$ is the maximal proper convex subgroup $C$ of $G$.
\end{lemma}
\begin{lemma}\label{l:fio}
    If $Q$ is torsion-free abelian and $Q' < Q$ has finite index, then the orders of $Q$ are in bijection with the orders on $Q'$.
\end{lemma}
\begin{proof}
    An order on $Q$ induces an order on $Q'$ by restriction.
    In the other direction, set $n = [Q:Q']$.
    For any ordering $<_{Q'}$ on $Q'$, define an ordering on $Q$ by
    \[
    x <_Q y \text{ if and only if } x^n <_{Q'} y^n .
    \]
    Then $<_Q$ is an ordering since $(xz)^n = x^nz^n$, and it is easy to check that these procedures are inverses of each other.
\end{proof}
Most of the time, we only need to consider orderings on a finitely generated free abelian group $Q$.
We will need the following general lemma about integral characters on such groups.
\begin{lemma}\label{l:character}
    Let $<$ be an order on $\zz^n$, and $y_1 < y_2 < \dots < y_k$ a finite collection of elements of $\zz^n$.
    Then there is a character $\gf: \zz^n \to \zz$ such that $\gf(y_1) < \gf(y_2) < \dots <\gf(y_k)$.
\end{lemma}
\begin{proof}
    We induct on $n$.
    There is an order-preserving nontrivial character to $\rr$.
    If this character is injective, then, since preserving the order of the $y_i$ is an open condition, we can perturb it slightly to land in $\qq$ and hence to have image isomorphic to $\zz$.
    Otherwise, the kernel is a nontrivial proper convex subgroup $K$, and we get an exact sequence
    \[
    0 \to K \to \zz^n \to Q \to 0
    \]
    of ordered free abelian groups and order-preserving maps.

    After choosing a splitting, write $\zz^n = Q \oplus K$, and decompose our elements $y_i$ as $(q_i, k_i)$.
    Independent of the splitting we have
    \[
    y_i < y_j \text{ if and only if } q_i < _Q q_j \text{ or } q_i = q_j \text{ and } k_i <_K k_j.
    \]
    By induction, there are characters $\gf_Q: Q \to \zz$ and $\gf_K: K \to \zz$ which strictly preserve the orderings of the $q_i$ and $k_i$ respectively.

    Our homomorphism $\gf$ is defined as
    \[
    \gf = N \gf_Q + \gf_K \text{ for } N \gg 0.
    \]

    If $(q_i, k_i) < (q_j, k_j)$, we need to check that $\gf(q_i, k_i) < \gf(q_j, k_j)$.
    If $q_i = q_j$, then $k_i <_K k_j$, so we are done since $\gf_K(k_i) < \gf_K(k_j)$.
    If $q_i <_{Q} q_j$ then $N\gf_Q(q_i) <_Q N\gf_Q(q_j)$ and adding $\gf_K$ preserves this inequality for $N$ sufficiently large.
\end{proof}

\subsection{Malcev--Neumann crossed products} Suppose that $Q$ is orderable, fix an order $<$ on $Q$, and suppose we have a crossed product $R*Q$.
Let $R*_{<}Q$ denote the set of functions $Q \to R$ with well-ordered support.
We will often think of such a function $f$ as a formal series $f=\sum_{q \in Q} f_{q}q$ with coefficients in $R$.
The \emph{leading coefficient} of $f$ is the value of $f$ on the minimal element $q_{\min}$ of its support, and the \emph{leading term} is the term $f_{q_{\min}} q_{\min}$.
\begin{lemma}\label{l:woprop}
    \cite{p89}*{Lemmas 2.9, 2.10} The well-ordered subsets of $Q$ have the following properties:
    \begin{enumerate}
        \item\label{i:subset}
        Any subset of a well-ordered subset is well-ordered.
        \item\label{i:union}
        The union of two well-ordered subsets is well-ordered.
        \item\label{i:product}
        The product (using the multiplication in $Q$) of two well-ordered subsets is well-ordered.
        \item\label{i:finiteness}
        Every element of the product has only finitely many representations as a product of two elements from the subsets.
        \item\label{i:powers}
        If $P>1$ is a positive well-ordered subset, then $\bigcup_{n=1}^{\infty} P^{n}$ is well-ordered.
    \end{enumerate}
\end{lemma}

Properties \eqref{i:subset} and \eqref{i:union} imply that $R*_{<}Q$ is an abelian group under point-wise addition.
Property \eqref{i:finiteness} allows one to take the convolution of two such functions using formula \eqref{e:conv}, and \eqref{i:product} and \eqref{i:subset} show that $R*_{<}Q$ is closed under this convolution, so $R*_{<}Q$ is a ring, which we will call the \emph{Malcev--Neumann crossed product}.
Note that $R*Q$ is a subring of $R*_{<}Q$.

Novikov rings, as defined in the introduction, are special cases of Malcev--Neumann crossed products by choosing $Q=\im \gf$, $K=\ker \gf$ and $R=\FK$, for $\gf : G \to \rr$ a character.

\subsection{Extensions and subgroups} Suppose now that $K \nsgp G \onto Q$ with $Q$ orderable.
Suppose that $R$ is a ring containing $\FK$ such that the conjugation $G$-action on $\FK$ extends uniquely to $R$; we say such a ring is $G$-invariant.
\begin{lemma}\label{l:MN}
    For any order $<$ on $Q$,
    \begin{enumerate}
        \item\label{i:emb}
        The group ring $\FG=\FK*Q$ embeds into a Malcev--Neumann crossed product $R*_{<} Q$.
        \item\label{i:lead}
        If $R$ has no zero divisors, then an element $f \in R*_{<} Q$ is invertible if and only if the leading coefficient of $f$ is a unit in $R$.
        \item\label{i:field}
        If $R$ is a $G$-invariant skew field, then $R*_{<}Q$ is a skew field.
    \end{enumerate}
\end{lemma}
\begin{proof}
    \eqref{i:emb} Since $R$ is invariant we can change the range of the structure maps $\gt$ and $\mu$ of the crossed product $\FG=\FK*Q$ to $\Aut(R)$ and $R^{\times}$, respectively, and they will still satisfy the required equations.
    Thus we obtain a crossed product $R*Q > \FG$, and then we further extend to the Malcev--Neumann crossed product $R*_{<} Q$.

    \eqref{i:lead} Since $R$ has no zero divisors, the leading term of a product $fg$ is the product of the leading terms of $f$ and $g$, so if $f$ is invertible then the leading term and hence the leading coefficient are units.
    Conversely, if the leading coefficient is a unit then the leading term is invertible, so we can factor it out from $f$ to obtain a series of the form $1-P$ for $P$ a series with positive support, and the inverse of that is the geometric series $1+P+P^2+\dots$ by Lemma \ref{l:woprop} \eqref{i:powers}.

    \eqref{i:field} Follows from \eqref{i:lead}.
\end{proof}
We have the following version of Lemma~\ref{l:decomp}.
\begin{lemma}\label{l:MNdecomp}
    Let $Q' \nsgp Q$ be a normal subgroup.
    Suppose $Q$ is ordered, $Q'$ has the induced order, and we are given a crossed product $R*Q$.
    \begin{enumerate}
        \item If $Q/Q'$ is finite, then
        \[
        R*_{<}Q \cong (R*_{<}Q') * Q/Q'.
        \]
        \item If $Q/Q'$ has an order such that $Q\to Q/Q'$ is order-preserving, then
        \[
        R*_{<}Q \cong (R*_{<}Q') *_{<} Q/Q'.
        \]
    \end{enumerate}
\end{lemma}
\begin{proof}
    This is immediate from the form of the structure maps for $R*Q = (R*Q')*Q/Q'$ in Lemma~\ref{l:decomp} and the fact that a subset of $Q$ is well-ordered if and only if its intersection with each of the finitely many cosets (in the first case) is well-ordered, or (in the second case) both the intersections with cosets and the projection to $Q/Q'$ are well-ordered.
\end{proof}
\begin{lemma}\label{l:sum}
    Suppose $G=K\times Q$, $Q$ is ordered, $K$ is finite, and we are given a crossed product $R*G$.
    Then
    \[
    (R*K)*_{<}Q \cong (R*_{<}Q)*K.
    \]
\end{lemma}
\begin{proof}
    By Lemma~\ref{l:decomp} $(R*K)*Q=R*G \cong (R*Q)*K$.
    The order on $Q$ induces an order on each $Q$-coset.
    By combining terms both $(R*K)*_{<}Q$ and $(R*_{<}Q)*K$ can be identified with the ring of functions on $G$ whose support intersects each coset in well-ordered set.
\end{proof}

We have the following versions of the exponential law for finite index subgroups.
\begin{lemma}\label{l:p2}
    Suppose $K \nsgp G \onto Q$ where $Q$ is ordered, $H \nsgp G$ has finite index, and $HK=G$.
    Let $K'=H\cap K$.
    Given a crossed product $R * G/K'$, for any order on $Q$,
    \[
    (R *_{<} Q)*G/H \cong (R*G/H)*_{<}Q .
    \]
\end{lemma}
\begin{proof}
    Both $K/K' \cong HK/H \cong G/H$ and $H/K' \cong HK/K \cong Q$ are normal subgroups of $G/K'$ and they intersect trivially, so they commute, and since $HK=G$, we have $G/K' \cong K/K'\times H/K'$.
    So, Lemma~\ref{l:sum} implies $(R * K/K') *_{<} Q \cong (R *_{<} Q)* K/K'$, and substituting $K/K' \cong G/H$ proves the claim.
\end{proof}
\begin{lemma}\label{l:p3}
    Suppose $K \nsgp G \onto Q$ where $Q$ is ordered, and $H \nsgp G$ has finite index.
    Set $K'=H\cap K$ and $Q'= H/K$, and consider the induced order on $Q'$.
    Given a crossed product $R' * G/K'$, set $R=R'*K/K'$.
    Then
    \[
    (R' *_{<} Q')*G/H \cong R*_{<}Q.
    \]

    Furthermore, this isomorphism restricts to an isomorphism
    \[
    (R' * Q')*G/H \cong R*Q.
    \]
    which does not depend on the ordering on $Q$.
\end{lemma}
\begin{proof}
    We have:
    \begin{align*}
        & R*_{<} Q \cong (R *_{<} Q')*Q/Q' &\text{by Lemma \ref{l:MNdecomp}(1)} \\
        &\cong \bigl((R' *HK/H)*_{<} Q'\bigr)*G/HK &\text{since $K/K' \cong HK/H$ and $Q/Q' \cong G/HK$} \\
        &\cong\bigl((R' *_{<} Q')*HK/H \bigr)*G/HK &\text{by Lemma \ref{l:p2} with $G=HK$ and $K=K'$} \\
        &\cong (R' *_{<} Q')*G/H &\text{by Lemma \ref{l:decomp}} \\
    \end{align*}

    The second statement follows from the fact that each of the isomorphisms between Malcev--Neumann crossed products in Lemma \ref{l:MNdecomp}{(1)} and Lemma \ref{l:p2} are induced by isomorphisms between the corresponding crossed products.
\end{proof}

\section{Linnell skew fields}\label{s:linnell}
Suppose a group ring $\FG$ is contained in a skew field $D$, and $D$ is the division closure of $\FG$ in $D$ (i.e.
there is no proper skew subfield $D'$ with $\FG < D'<D$).
We shall need the fact that two homomorphisms from $D$ which agree on $\FG$ are equal.

Following \cite{jl23} we say $D$ is \emph{Linnell} if for any subgroup $K < G$, $g_1, \dots, g_n \in G$ in different $K$-cosets, and $d_1, \dots, d_n$ nonzero elements in the division closure of $\FK$ in $D$, the sum $\sum^n_{i = 1}d_i g_i$ is nonzero.

Hughes \cite{h70} showed that if $G$ is locally indicable then such a skew field, if it exists, is unique up to isomorphisms fixing $\FG$ (Hughes actually requires a weaker linear independence condition; as explained in \cite{jl23}, it follows from Gr\"{a}ter's work in \cite{g20} that for locally indicable groups they are equivalent).
We shall denote any Linnell skew field containing $\FG$ by $D_{\FG}$.
\begin{lemma}\label{l:DG}
    Suppose $G$ is locally indicable and $D_{\FG}$ exists.
    \begin{enumerate}
        \item\label{i:inva}
        (Hughes \cite{h70}*{p.~183}) The action of $\Aut(G)$ on $\FG$ extends to an action on $D_{\FG}$, in particular, $D_{\FG}$ is a $G$-invariant $\FG$-skew field.
        \item\label{i:subgroup}
        If $K$ is a subgroup of $G$, then the division closure of $\FK$ in $D_{\FG}$ is $D_{\FK}$.
        \item\label{i:nsubgroup}
        If $K$ is a normal subgroup of $G$, then $\FG < D_{\FK} *G/K < D_{\FG}$.
        \item\label{i:fi}
        Suppose $H$ is a finite index normal subgroup of $G$.
        Then:
        \[
        D_{\FG} \cong D_{\FH}*G/H.
        \]
    \end{enumerate}
\end{lemma}
\begin{proof}
    \eqref{i:inva} Uniqueness of $D_{\FG}$ implies that every automorphism of $\FG$ extends to an automorphism of $D_{\FG}$.
    Since $D_{\FG}$ is a division closure of $\FG$ such an extension is unique, and the claim follows.

    \eqref{i:subgroup} It is easy to see that the division closure of $\FK$ is Linnell, so this follows from uniqueness.

    \eqref{i:nsubgroup} By \eqref{i:inva}, we can extend the structure maps for $\FK*G/K$ to form a crossed product $D_{\FK} *G/K$.
    Since $D_{\FG}$ is Linnell, $D_{\FK} *G/K$ injects into $D_{\FG}$.

    \eqref{i:fi} By \eqref{i:nsubgroup} $\FG < D_{\FH} *G/H < D_{\FG}$.
    Thus $D_{\FH} *G/H$ is a finite dimensional algebra over $D_{\FH}$ with no zero divisors, so it is a skew field, and hence isomorphic to $D_{\FG}$.
\end{proof}

Note that Lemma \ref{l:op} implies that orderable groups are locally indicable.
If $G$ is orderable, then the division closure of $\FG$ in the Malcev-Neumann skew field $\F\ast_<G$ is the Linnell skew field $D_{\FG}$.

If $Q$ is an amenable group, $D$ is a skew field, and $D*Q$ has no zero divisors, then $D*Q$ satisfies the Ore condition with respect to all non-zero elements \cite{t54}, and the associated Ore localization $\Ore(D*Q)$ is a skew field.
Elements of $\Ore(D*Q)$ can be thought of as fractions $f/g$ with $f,g \in D*Q$, the Ore condition ensures the existence of common denominators, and therefore the usual rules for equality and arithmetic operations apply.
$\Ore(D*Q)$ is clearly the division closure of $D*Q$, moreover, any embedding of $D*Q$ into a skew field $S$ induces an isomorphism from $\Ore(D*Q)$ to the division closure of $D*Q$ in $S$.
If $D*Q \cong \F Q$, then $\Ore(D*Q)$ is a Linnell skew field of $\F Q$.
\begin{lemma}\label{l:Ore}
    Suppose $K \nsgp G \onto Q$ with $G$ locally indicable and $Q$ orderable and amenable.
    If $D_{\FG}$ exists then:
    \begin{enumerate}
        \item\label{i:ore}
        $D_{\FG} \cong \Ore(D_{\FK}*Q)$.
        \item\label{i:div}
        If $D_{\FK}*Q$ is contained in a skew field $S$, then $D_{\FG}$ is the division closure of $\FG$ in $S$.
        \item\label{i:MNDG}
        For any order on $Q$, $D_{\FK}*_{<}Q$ is a skew field and $D_{\FG}$ is the division closure of $\FG$ in $D_{\FK}*_{<}Q$.
    \end{enumerate}
\end{lemma}
\begin{proof}
    For \eqref{i:ore}, by Lemma~\ref{l:DG}\eqref{i:nsubgroup}, $D_{\FK}*Q$ injects into $\DG$; and since $\DG$ is a skew field this extends to an injective homomorphism $\Ore(D_{\FK}*Q)\to D_{\FG}$.
    Since the image of $\Ore(D_{\FK}*Q) \to D_{\FG}$ contains $\FG$, it is surjective.
    \eqref{i:div} follows immediately from \eqref{i:ore}.
    Finally, $D_{\FK}*_{<}Q$ is a skew field by Lemma~\ref{l:MN}\eqref{i:field}, which contains $D_{\FK}*Q$, so \eqref{i:MNDG} follows from \eqref{i:div}.
\end{proof}
So, under the hypothesis of the above lemma, $D_{\FK}*_{<}Q$ contains two subrings, $\FK*_{<} Q$ and $D_{\FG}$.
In particular, if $\psi: G \onto \zz$ is a character and $J= \ker \psi$, then $D_{\FJ} *_{<} \zz$ contains $D_{\FG}$ and the Novikov ring $\F G^\psi=\FJ *_{<} \zz$.
In other words, $D_{\FG}$ contains two subrings, $(\FK*_{<} Q ) \cap D_{\FG} $ and $\F G^\psi \cap D_{\FG}$.
The next lemma shows that an element of the former lies in the latter for a suitably chosen character.
\begin{lemma}\label{l:nov}
    Suppose $K \nsgp G \onto Q$ with $G$ locally indicable and $Q$ orderable amenable, and $D_{\FG}$ exists.
    Fix any order $<$ on $Q$.
    Suppose $f \in (\FK*_{<} Q ) \cap D_{\FG} $ and write $f=x/y$ with $x,y \in D_{\FK}*Q$.
    Suppose $\gf:Q \to \zz$ is a character which strictly preserves the order on the support of $y$, and let $\psi:G \to \zz$ be the induced character.
    Then $f \in \F G^\psi$.
\end{lemma}
\begin{proof}
    Let $J=\ker \psi$.
    By Lemma~\ref{l:DG}\eqref{i:nsubgroup} $D_{\FK}*J/K < D_{\FJ}$, so by Lemma~\ref{l:decomp} $D_{\FK}*Q \cong (D_{\FK}*J/K)* \zz < D_{\FJ} * \zz$.
    Since $x,y \in D_{\FK}*Q$, the same fraction $x/y$ represents $f$ in $D_{\FG}=\Ore(D_{\FJ}*\zz)$.

    We can assume that $y=1-P$ where $P \in D_{\FK}*Q$ has positive support in $Q$.
    Then $P$ considered as an element of $D_{\FJ} * \zz$ has finite and positive support in $\zz$, because $\gf$ preserves the order on the support of $y$.
    Since $f$ is represented by the series $x(1+P+P^2+\dots)$ in $\FK*_{<} Q $, the same series represents $f$ in $D_{\FJ} *_{<} \zz$.

    We want the series to be in $\FJ *_{<} \zz = \FG^\psi$, so we need the support of this series to have finite intersection with $\gf^{-1}(z)$ for each $z\in \zz$ and to be contained in $\gf^{-1}([t,\infty))$ for some~$t$.
    Let $p_{\min} $ and $p_{\max}$ be the minimal and maximal elements of the support of $P$.
    Then the support of $P^{n}$ is between $p_{\min}^{n} $ and $p_{\max}^{n}$, hence its image under $\gf$ is between $n \gf(p_{\min}) $ and $n \gf(p_{\max})$, and since $0< \gf(p_{\min}) \leq \gf(p_{\max})$ and the supports of $P^{n}$ and $x$ are finite, the claim follows.
\end{proof}

\section{Invariant Malcev--Neumann rings} We now consider the following setup.
Suppose we have an ambient group $\cG$ which is locally indicable and such that $D_{\F\cG}$ exists, and suppose we have two normal subgroups $K < G$ of $\cG$ with $Q = G/K$ orderable and amenable.
Let $R$ be a subring of $D_{\FK}$ which contains $\F K$ and is invariant under the $\cG$-action on $D_{\FK}$ induced by conjugation on $\FK$.
So, for each order on $Q$, we have two subrings of $D_{\FK}*_{<}Q$: $R*_{<} Q$ (since $R<D_{\FK}$) and $D_{\FG}$ (by Lemma~\ref{l:Ore}\eqref{i:MNDG}).
Note that Malcev--Neumann rings $D_{\FK}*_{<}Q$ and $R*_{<} Q$ are generally not $\cG$-invariant; the conjugation action of $\cG$ on $Q$ does not respect the order, and destroys the well-ordering property of supports of series.

The intersection $(R*_{<} Q) \cap D_{\FG}$ is obviously a subring of $\DG$ and we intersect all such subrings over all orders on $Q$.
We denote the resulting subring by $R**Q$, and call it the \emph{invariant} Malcev--Neumann ring.
It has the same properties as $R$: it contains $\FG$ and is invariant under the $\cG$-action on $D_{\FG}$.

We can think of $R**Q$ as consisting of elements of $D_{\FG}$ which, for each order on $Q$ are represented by a Malcev--Neumann series with $R$ coefficients.

Given a non-empty subset $S$ of $R**Q-0$, denote by $\ell(S)$ the collection of all leading coefficients which appear in these series for elements of $S$ for all orders.
\begin{lemma}\label{l:inv}
    Let $f\in R**Q$.
    Then:
    \begin{enumerate}
        \item\label{i:fin}
        $\ell(f)$ is finite.
        \item\label{i:inv}
        $f$ is invertible in $R**Q$ if and only if every $r \in \ell(f)$ is a unit in $R$.
    \end{enumerate}
\end{lemma}
\begin{proof}
    \eqref{i:fin} By Lemma~\ref{l:Ore}\eqref{i:ore} $R**Q < \Ore(D_{\FK}*Q) = D_{\FG}$, so we can write $f=xy^{-1}$ with $x, y \in D_{\FK}*Q$.
    For each order $<$ on $Q$ we can write $y^{-1}$ as a Malcev--Neumann series in $D_{\FK}*_{<}Q$.
    Since the support of $y$ is finite, we only obtain finitely many series, depending only on the leading term of $y$ with respect to the order.
    Left-multiplying such a series by $x$ gives a series for $f$ in $D_{\FK}*_{<}Q$, and since $f \in R**Q$ this series has $R$-coefficients.
    So the set of the leading coefficients of these series is $\ell(f)$.

    \eqref{i:inv} Follows from Lemma~\ref{l:MN}\eqref{i:lead} and the fact that $R < D_{\F K}$ implies that $R$ has no zero divisors.
\end{proof}

For torsion-free abelian quotients we have an analogue of Lemma \ref{l:p3} for invariant rings.
\begin{lemma}\label{l:cp3}
    Suppose $K \nsgp G \onto Q$ with $Q$ torsion-free abelian, and $H \nsgp G$ has finite index.
    Let $K'=H\cap K$, $Q'= H/K' < Q$.
    Given a $\cG$-invariant ring $\FK' < R' < D_{\FK'}$, set $R = R'*K/K'$.
    Then:
    \[
    (R' ** Q')*G/H \cong R**Q.
    \]
\end{lemma}
\begin{proof}
    We first prove that for a fixed order on $Q$ and the induced order on $Q'$
    \[
    ((R' *_{<} Q') \cap D_{\FH})*G/H \cong (R *_{<} Q) \cap D_{\FG},
    \]

    Here, the intersection on the left hand side is taken in $D_{\F K'} *_{<} Q'$, and by Lemma~\ref{l:Ore}\eqref{i:MNDG}, $D_{\F H}$ is the division closure of $\im(i')$, where $i': D_{\F K'} *Q' \to D_{\F K'} *_{<} Q'$ is the natural embedding.
    Similarly the intersection on the right hand side is taken in $D_{\F K} *_{<} Q$, and $D_{\F G}$ is the division closure of $\im(i)$, where $i: D_{\F K} *Q \to D_{\F K} *_{<} Q$.

    Since the constructions in Section~\ref{s:cross} are functorial with respect to the coefficient ring, we have the following commutative diagram:
    \begin{center}
        \begin{tikzcd}
            (R' *_{<} Q')*G/H \arrow[d, hookrightarrow ] \arrow[r,"\cong \ref{l:p3} "] &(R' * K/K')*_{<} Q \arrow[r, " ="] \arrow[d, hookrightarrow] & R *_{<} Q \arrow[d, hookrightarrow]\\
            (D_{\F K'} *_{<} Q')*G/H \arrow[r,"\cong \ref{l:p3}"] &(D_{\F K'}*K/K')*_{<} Q \arrow[r,"\cong \ref{l:DG}\eqref{i:fi}"] & D_{\F K} *_{<} Q\\
            (D_{\F K'} *Q') *G/H \arrow[u, "i' * id", hookrightarrow] \arrow[r,"\cong \ref{l:p3}"] &(D_{\F K'}*K/K')*Q \arrow[u, hookrightarrow] \arrow[r,"\cong \ref{l:DG}\eqref{i:fi}"] & D_{\F K}*Q \arrow[u, "i", hookrightarrow]
        \end{tikzcd}
    \end{center}
    The composition of the maps in the middle row maps the division closure of $\im( i' * id) $ isomorphically onto the division closure of $\im( i) $, thus we have
    \[
    ((R' *_{<} Q')*G/H) \cap \Div(\im( i' * id)) \cong (R *_{<} Q) \cap D_{\F G}.
    \]

    Since the division closure of $\im i'$ is $D_{\F H}$, $i' * id$ factors through $D_{\F H}*G/H$, and the division closure of $\im( i' * id) $ contains $D_{\F H}*G/H$.
    Since $D_{\F H}*G/H$ is a skew field by Lemma~\ref{l:DG}\eqref{i:fi} the division closure of $\im( i' * id)$ is $D_{\F H}*G/H$.
    Therefore, we have the desired formula:
    \[
    ((R' *_{<} Q') \cap D_{\FH})*G/H = ((R' *_{<} Q' )*G/H) \cap (D_{\FH} * G/H) \cong (R *_{<} Q) \cap D_{\FG}.
    \]

    Next, we note that the bottom isomorphism $(D_{\F K'} *Q') *G/H \rightarrow D_{\FK} *Q$ is independent of the order; and hence so is the isomorphism between division closures $D_{\F H} * G/H \rightarrow D_{\FG}$.
    Since by Lemma~\ref{l:fio} orders on $Q$ and $Q'$ are in one-to-one correspondence, we get
    \[
    (R' ** Q') * G/H \cong R ** Q.
    \]
\end{proof}

\section{Inductive rings} Suppose that $ \dots \nsgp K_{i} \nsgp K_{i-1} \nsgp \dots \nsgp K_{0} =G$ is a residual chain of normal subgroups, such that each $Q_{i}=K_{i}/K_{i+1}$ is orderable and amenable, and there exists a Linnell skew field $\DG$ over $\FG$.
Since $\FK_{i}$ is $G$-invariant, we can apply the invariant ring construction repeatedly starting from the ring $\FK_{i}$ for any $i$.
Define inductively $R_{i}^{i}=\FK_{i}$ and for $0\leq j<i$, $R_{i}^{j}=R_{i}^{j+1}**Q_{j}< D_{\FK_{j}}$.
Since
\[
\FK_{i}= \FK_{i+1}*Q_{i} < \FK_{i+1}**Q_{i} ,
\]
we have $R_{i}^{i}< R_{i+1}^{i}$, and it follows by induction that $R_{i}^{j}< R_{i+1}^{j}$ for all $0\leq j\leq i$.

In particular, $\FG=R_{0}^{0}< R_{1}^{0}<\dots < \DG$ is an increasing sequence of subrings of $\DG$.
Let $U= \bigcup_{i=0}^{\infty} R_{i}^{0}$ denote their union.
\begin{Theorem}\label{t:union}
    \[
    U=\DG.
    \]
\end{Theorem}
\begin{proof}
    Since $\FG \subset U$, it is enough to show that any nonzero element of $U$ is invertible.
    Let $f \in U-0$ be arbitrary, so $f \in R_{i}^{0}$ for some $i$.
    We can repeatedly apply the leading coefficients operator $\ell$: $\ell(f) \subset R_{i}^{1}$, $\ell^{2}(f) \subset R_{i}^{2}$, etc., until we reach $R_{i}^{i}$, by Lemma~\ref{l:inv}(1) $\ell^{i}(f)$ is a finite subset of the group ring $\FK_{i}$.

    Since the subgroups $K_i$ are residual in $G$, there exists $n \geq i $ such that the group elements in the support of the elements of $\ell^{i}(f)$ are in different $K_n$-cosets.
    Thinking of $f$ as an element of $R_{n}^{0}$, this implies that the elements of $\ell^{n}(f) \subset \FK_{n}$ are supported on single group elements.
    Hence they are units in $\FK_n$, so $f$ is a unit in $U$ by inductive application of Lemma~\ref{l:inv}(2).
\end{proof}

\subsection{Finite index subgroups}
\begin{Theorem}\label{t:finiteindex}
    Let $ \dots \nsgp K_{i} \nsgp K_{i-1} \nsgp \dots \nsgp K_{0} =G$ be a residual chain of normal subgroups, such that $Q_{i}=K_{i}/K_{i+1}$ is torsion-free abelian and there exists a Linnell $\DG$ over $\FG$.
    Let $H \nsgp G$ be finite index and $ \dots \nsgp K_{i}' \nsgp K_{i-1}' \nsgp \dots \nsgp K_{0}' =H$ the corresponding chain where $K_i' = K_i \cap H$.
    For any $i, j$ with $0 \leq j \leq i$ let $R_{i}^{j}$ and $R_{i}^{\prime j}$ be the inductive rings for $G$ and $H$ respectively.
    Then we have
    \[
    R_{i}^{j} \cong R_{i}^{\prime j} \ast K_j/K_j'.
    \]
    In particular, $R_{i}^{0} \cong R_{i}^{\prime 0} \ast G/H$.
\end{Theorem}
\begin{proof}
    We backwards induct on $j$.
    For $j = i$ this is saying $\FK_i \cong \FK_i' \ast K_i/K_i'$, which is true by Lemma \ref{l:decomp}.
    Suppose now that $R_{i}^{j+1} \cong R_{i}^{\prime j+1} \ast K_{j+1}/K_{j+1}'$.
    Then an application of Lemma~\ref{l:cp3} with $G = K_j$, $H = K_j'$ implies that $R_{i}^{j} \cong R_{i}^{\prime j}\ast K_j/K_j'$.
\end{proof}

\section{Residual properties}\label{s:resid}
To construct our inductive rings, we need a residual chain of normal subgroups $K_i$ of $G$ whose successive quotients are orderable amenable.
To get the desired applications to virtual fiberings, we will need for each $K_i$ a finite index subgroup $H_i$ of $G$ such that $H_i/K_i$ is free abelian.
We show now that these conditions can be rephrased as a residual property of $G$, and that this residual property is equivalent to being RFRS.

Recall that a group $G$ is \emph{RFRS} if it has normal residual sequences of subgroups $H_i$ and $K_i$ with $K_{i} < H_{i}$, $H_i$ is finite index in $G$ and $H_{0}=K_{0}=G$, and $H_{i}/K_{i+1}$ is torsion-free abelian.
The following two lemmas, and their use in the following Theorem \ref{t:rvali}, were kindly explained to us by Sam Fisher and Dawid Kielak.
\begin{lemma}\label{l:cs}
    A finitely generated virtually abelian locally indicable group $Q$ is RFRS.
\end{lemma}
\begin{proof}
    Since $Q$ is locally indicable, there exists an epimorphism $\gf:Q \to \zz$.
    After choosing a section of $\gf$, we identify $\zz$ with a subgroup of $Q$.
    Let $G:=\ker \gf$, then $Q = G \zz$.
    Note that $G$ is also finitely generated virtually abelian locally indicable, of smaller cohomological dimension, and is stabilized by the conjugation action of $\zz$.
    By induction on cohomological dimension, we can assume that $G$ is RFRS. Therefore $G$ has RFRS chains $H_i$ and $K_i$, and an easy exercise shows that we can assume each $H_i$ and $K_i$ is characteristic in $G$, and hence preserved by the action of $\zz$.

    Let $A \nsgp Q$ be a normal abelian subgroup of finite index.
    Then $A \cap \zz$ is a finite index subgroup of $\zz$ which acts trivially on $A$.
    Since $G/H_{i}$ is finite, we can choose a residual chain $B_{i}< A \cap \zz$, such that $B_{i}$ acts trivially on $G/H_{i}$ for each $i \geq 0$.

    Define $H'_{0}=K'_{0}:=Q$, and $H'_{i+1}:=H_{i} B_{i}$ and $K'_{i+1}:=K_{i}$ for $i \geq 0$.
    We claim that $H'_{i}$ and $K'_{i}$ form RFRS chains for $Q$.
    Obviously, each $H'_{i}$ has finite index in $Q$ and contains $K'_{i}$, each $K'_{i}$ is normal in $Q$, and $H'_{0}/K'_{1}=Q/G=\zz$.

    We now check that $H'_{i}$ is normal in $Q$ for $i \geq 1$.
    Any element of $Q$ can be written as a product $z g$ with $z \in \zz$ and $g \in G$.
    Then we have
    \begin{align*}
        z g H_{i} B_{i} &= z B_{i} g H_{i} \text{ since $B_{i}$ acts trivially on $G/H_{i}$},\\
        &= B_{i} z g H_{i} \text{ since $B_{i} < \zz$},\\
        &= B_{i} H_{i} z g \text{ since $H_{i}$ is normal in $Q$},\\
        &= H_{i} B_{i} z g.
    \end{align*}

    Next we show that $H'_{i+1} / K'_{i+2}$ is torsion-free abelian for $i \geq 0$.
    Since $A \cap H_{i}$ has finite index in $H_{i}$, the image of $A \cap H_{i}$ has finite index in $H_{i}/K_{i+1}$.
    Since $A \cap \zz$ acts trivially on $A \cap H_{i}$, and $H_{i}/K_{i+1}$ is finitely generated free abelian, it follows that $A \cap \zz$, and therefore $B_{i}$, act trivially on $H_{i}/K_{i+1}$.
    Thus $H'_{i+1} / K'_{i+2} = H_{i}B_{i} / K_{i+1} = (H_{i} / K_{i+1}) \times B_{i}$ is torsion-free abelian.

    To finish the proof we note that $H'_{i}$ is a residual sequence since both $H_{i}$ and $B_{i}$ are.
\end{proof}
\begin{lemma}\label{l:rrfrs}
    Countable residually RFRS groups are RFRS.
\end{lemma}
\begin{proof}
    Given a sequence of RFRS groups $G^{j}$ with RFRS chains $H^{j}_{i}$, $K^{j}_{i}$, the countable direct product $G=G^{1} \times G^{2} \times \cdots$ is RFRS with the RFRS chains
    \[
    H_{i}:=H^{1}_{i} \times H^{2}_{i-1} \times \dots \times H^{i}_{1} \times G^{i+1} \times G^{i+2} \dots
    \]
    and
    \[
    K_{i}:=K^{1}_{i} \times K^{2}_{i-1} \times \dots \times K^{i}_{1} \times G^{i+1} \times G^{i+2} \dots
    \]
    Since subgroups of RFRS groups are RFRS, the claim follows.
\end{proof}
\begin{Theorem}\label{t:rvali}
    Let $G$ be a finitely generated group.
    The following conditions are equivalent:
    \begin{enumerate}
        \item\label{i:chain}
        There are sequences of normal subgroups $K_i$ and $H_i$ of $G$ with $H_0 = K_0 = G$ such that $K_i$ is a residual chain, $H_i$ has finite index and contains $K_{i}$, and $K_i/K_{i+1}$ and $H_i/K_i$ are free abelian of finite rank.
        \item\label{i:rvali}
        $G$ is residually (virtually abelian locally indicable).
        \item\label{i:rfrs}
        $G$ is RFRS.
    \end{enumerate}
\end{Theorem}
\begin{proof}
    The implication \eqref{i:chain}$\implies$\eqref{i:rvali} is clear.
    The implication \eqref{i:rvali}$\implies$\eqref{i:rfrs} is immediate from Lemmas~\ref{l:cs} and \ref{l:rrfrs}.
    To prove \eqref{i:rfrs}$\implies$\eqref{i:chain}, note that for RFRS chains $H_{i}$, $K_{i}$ we have $K_i/K_{i+1} \nsgp H_{i}/K_{i+1}$.
\end{proof}

\section{Fibering and characters}\label{s:fc}
We say that a complex $C_{*}$ of free $R$-modules has \emph{type $FP_{n}(R)$} if it is $R$-chain homotopy equivalent to a complex of free $R$-modules which have finite rank in degrees $\leq n$.
For a right $R$-module $M$ we denote $H_{*}(C_{*}; M):=H_{*}(M \otimes_{R} C_{*})$ and if $M=D$ happens to be a skew field we define $D$-Betti numbers of $C_{*}$ by $b_{k}(C_{*}; D):=\dim_{D} H_{k}(C_{*}; D)$.

If $Y$ is a $G$-$CW$-complex and $R=\zz G$ we define the equivariant homology of $Y$ with $M$-coefficients: $H_{*}(Y; M):=H_{*}(C_{*}(Y); M)$.

We can now prove our main theorem connecting $D_{\F G}$-Betti numbers and Novikov homology in finite covers.
Note that Jaikin-Zapirain \cite{j21}*{Corollary 1.3} showed that residually (amenable locally indicable) groups have $D_{\FG}$ over any field $\F$, and $D_{\F G}$ is universal in the sense that for any other $\F G$-skew field $D$, $b_{*}(Y; \DG) \leq b_{*}(Y; D)$.
In particular, $D_{\FG}$ exists for any RFRS group.
This leads to the following:
\begin{Theorem}\label{t:vc}
    Suppose that $G$ is finitely generated RFRS and $C_{*}$ is a finite chain complex of finitely generated free $\FG$-modules.
    Then there is a finite index subgroup $H$ and a character $\psi: H \to \zz$ such that
    \[
    H_{i}(C_{*} ; \FH^{\pm\psi}) = (\FH^{\pm\psi})^{[G:H]b_i(C_{*}; D_{\FG})} \text{ for all } i.
    \]
\end{Theorem}
\begin{proof}
    The theorem follows quickly from the following:
    \subsection*{Claim} Given any finite collection of elements $f_1, \dots, f_n$ in $D_{\FG}$, there is a finite index subgroup $H$ and a character $\psi: H \to \zz$ such that for any section $s:G/H \to G$ and for each $f_k = \sum_{g \in G/H} f_k^g s_{g} \in D_{\FH}*G/H$, the terms $f_k^g$ are in the Novikov ring of $\pm \psi$.
    In particular, the matrix of the right multiplication by $f_{k}$ on $D_{\FG}$, thought of as a linear map over $D_{\FH}$ with respect to the basis $(s_{g})$, has entries in the Novikov rings of $\pm \psi$.
    \begin{proof}
        There exists $m$ such that each $f_k$ is in $R_{m}^{0}$ by Theorem \ref{t:union}.
        Let $K_{i}'=H_{m} \cap K_{i}$, $Q'_{i}=K'_{i} / K'_{i+1}$ and let $R_{m}^{\prime 0}$ be the corresponding inductive ring for $H_{m}$.
        Note that $Q'_{i} < Q_{i}$ is finite index, so the orders on $Q'_{i}$ and $Q_{i}$ are in 1-1 correspondence.
        Also note that $K'_{m}=K_{m}$.
        To simplify notation, set $H=H_{m}$, $K=K_{m}$ and $Q=H/K$.
        We have that $Q$ is finitely generated free abelian, and therefore $Q=\bigoplus_{i=0}^{m-1} Q'_{i}$.

        Note that the coefficients $f_{k}^{g}$ depend on the choice of the section $s: G/H \to G$, but whether they are in the Novikov ring does not.
        So we fix a section.

        Choose arbitrary orders on each of the $Q'_i$ and the induced dictionary order $<$ on $Q$.
        Then the projections $\bigoplus_{i=j}^{m-1} Q'_{i} \to \ Q'_{j}$ are order preserving.

        By induction, using the definition of inductive rings and applying Lemma \ref{l:MNdecomp}(2):
        \[
        R_{m}^{\prime 0} < \Bigl(\cdots\bigl(( \FK*_{<}Q'_{m-1}) *_{<}Q'_{m-2} \bigr)*_{<} \cdots \Bigr)*_{<} Q'_{0} = \FK *_{<} \bigoplus_{i=0}^{m-1} Q'_{i}= \FK *_{<} Q.
        \]
        Since, by Theorem \ref{t:finiteindex}, each of the $f_k^g$-terms in $D_{\FH}$ are in $R_{m}^{\prime 0}$, we have $f_k^g \in \FK *_{<} Q$.

        Furthermore, each $f_k^g$ is an element of $\Ore(D_{\FK} \ast Q)$, and hence can be expressed as a fraction $x^{g}_k/y^{g}_k$ with $x^{g}_k, y^{g}_k \in D_{\FK} \ast Q$.
        By Lemma \ref{l:character}, there is an integral character $\gf: Q \to \zz$ which strictly preserves the order on the union of supports of $y^{g}_{k}$.
        Let $\psi: H \to \zz$ be the induced character and let $J = \ker \psi$.
        Then by Lemma~\ref{l:nov} each $f_k^g \in \FJ *_{<} \zz=\FH^{\psi}$.

        Flipping the order on each $Q'_i$ reverses the order on the union of supports, as does $-\gf$.
        So $f_k^g \in \FH^{-\psi}$, as well.
    \end{proof}

    To finish the proof, we put the chain complex $ D_{\FG} \otimes_{\FG} C_{*}$ into Smith normal form, so the matrices of differentials have some number of 1's on the diagonal and 0's elsewhere.
    Since the chain complex is finite there are only finitely many $\DG$-elements in the change of basis matrices (and their inverses).
    After passing to a finite index subgroup $H$ as in the claim, we can assume that all the entries of the resulting matrices over $D_{\FH}$ lie in the Novikov rings of $\pm\psi: H \to \zz$.
    Therefore, as matrices over $\FH^{\pm\psi}$ these are still invertible, and bring $\FH^{\pm\psi} \otimes_{\FH} C_{*}$ into $[G:H]$ copies of the same normal form.
    This implies the theorem.
\end{proof}

If in the theorem we just assume that $C_{i}$ is a finitely generated free $\FG$-module for $i \leq n$, then since $D_{\FG}$ is a skew field, the image of $D_{\FG} \otimes_{\FG} \partial_{n+1}$ is a finite dimensional $D_{\FG}$-vector space.
Therefore, we can truncate $C_{*}$ at degree $(n+1)$ and remove some $\FG$ summands from $C_{n+1}$ to obtain a finite subcomplex of finitely generated free $\FG$-modules $C_{*}' < C_{*}$ with the same $D_{\FG}$-homology in degrees up to $n$.
For any $\FG$-module $M$, the induced map $i_{*}:H_{i}(C_{*}'; M) \to H_{i}(C_{*}; M)$ is an isomorphism for $i<n$ and an epimorphism for $i=n$, hence applying the theorem to $C_{*}'$ gives the following:
\begin{corollary}\label{c:computation}
    Suppose that $G$ is finitely generated RFRS and $C_{*}$ is a chain complex of type $FP_{n}(\FG)$.
    Then there are a finite index subgroup $H$ and a character $\psi: H \to \zz$ such that
    \begin{enumerate}
        \item $H_i(C_{*}; \FH^{\pm\psi}) = (\FH^{\pm\psi})^{[G:H]b_i(C_{*}; D_{\FG})} \text{ for } i < n$.
        \item $H_n(C_{*}; \FH^{\pm\psi})$ is a quotient of $(\FH^{\pm\psi})^{[G:H]b_{n}(C_{*}; D_{\FG})}$.
    \end{enumerate}
\end{corollary}

Combining this with \cite{fgs10}*{Theorem 8 and Proposition 5} we obtain:
\begin{corollary}\label{c:dom}
    Suppose that $G$ is finitely generated RFRS and $C_{*}$ is a chain complex of type $FP_n(\FG)$ with $H_i(C_{*}; D_{\FG}) = 0$ for $i \le n$.
    Then there is a finite index subgroup $H$ and a map $\psi: H \to \zz$ such that $C_{*} $ is of type $FP_n(\F\ker \psi)$.
\end{corollary}

We note that \cite{fgs10} deals only with $\zz$-coefficients, but \cite{f24}*{Theorem 5.3} shows the arguments work for any ring $R$.
We also note that for chain complexes of free $R$-modules the notion of \emph{finite $n$-type over $R$} used in \cite{fgs10} is equivalent to type $FP_{n}(R)$.

Taking $C_{*}$ to be the chain complex of the universal cover of a $K(G,1)$ gives:
\begin{corollary}
    Suppose that $G$ is RFRS, of type $FP_n(\F)$, and $H_i(G; D_{\FG}) = 0$ for $i \le n$.
    Then there is a finite index subgroup $H$ and a map $\psi: H \to \zz$ such that $\ker \psi$ is of type $FP_n(\F)$.
\end{corollary}

\section{Homology growth}

For a finite complex $X$ the \emph{upper} and \emph{lower $\F $-homology growth} are defined by
\begin{align*}
    \ub_{k}(X;\F )&:=\inf_{X'\to X}\left(\sup_{X''\to X'} \frac{b_k(X'';\F )}{\abs{X''\to X}}\right),\\
    \lb_{k}(X;\F )&:=\sup_{X'\to X}\left(\inf_{X''\to X'} \frac{b_k(X'';\F )}{\abs{X''\to X}}\right).
\end{align*}
where $\inf$ and $\sup$ are taken over all finite covers and $\abs{X'' \to X}$ denotes the degree of a cover.
(We don't require the covers to be regular, or even connected, but if $X$ is disconnected we require the degree to be the same over all components.) At this point, we do not know any example with $\lb_{k}(X;\F_p ) \neq \ub_{k}(X;\F_p)$.

We collect some properties of these quantities from \cite{aos24}*{Section 2}.
\begin{Prop}
    \leavevmode
    \begin{enumerate}
        \item $\lb_{k}(X; \F)$ and $\ub_{k}(X; \F)$ are homotopy invariant.
        \item $\lb_{k}(X; \F) \leq \ub_{k}(X; \F) \leq$ number of $k$-cells in $X$.
        \item If $X' \to X$ is a finite cover, then
        \[
        \lb(X'; \F) = \abs{X' \to X} \lb(X; \F) \quad \text{and} \quad \ub(X'; \F) = \abs{X' \to X} \ub(X; \F).
        \]
        \item\cite{aos24}*{Lemma 2.7} If $X=Y\amalg Z$, then
        \[
        \lb(X; \F)=\lb(Y; \F)+\lb(Z; \F) \quad \text{and} \quad \ub(X; \F)=\ub(Y; \F)+\ub(Z; \F).
        \]
        \item\cite{aos24}*{Lemma 2.8} If $X$ is connected, then $\lb_{k}(X; \F)$ and $\ub_{k}(X; \F)$ can be computed using only connected covers.
        \item\cite{aos24}*{Corollary 2.5} $\lb_{k}(X;\qq ) = \ub_{k}(X;\qq )$.
    \end{enumerate}
\end{Prop}

For any continuous map $f:X\to X$, the mapping torus $T_{f}$ is homotopy equivalent to a complex with at most twice the number of cells of $X$.
Since $T_{f^{n}}$ is homotopy equivalent to an $n$-fold cover of $T_{f}$, the properties (1)--(3) above imply that $\ub_{*}(T_{f}; \F)=0$.
In fact, if the chain complex of the universal cover, $C_{*}(\widetilde X; \F)$, is of type $FP_n(\F\pi_{1}X)$, then $\lb_{k}(X; \F)$ and $\ub_{k}(X; \F)$ are well defined and finite for $k\leq n$ and we have the following version of L\"uck's Mapping Torus Theorem:
\begin{Theorem}[\cite{aos24}*{Theorem 2.10}]\label{t:mappingtorustheorem}
    Let $X$ be a complex with $C_{*}(\widetilde X; \F)$ of type $FP_n(\F\pi_{1}X)$, $f:X\to X$ a self-homotopy equivalence and $T_f$ its mapping torus.
    Then for $k\leq n$
    \[
    \ub_k(T_f;\F )=0.
    \]
\end{Theorem}
The universality of $D_{\F G}$ leads to the following:
\begin{Theorem}\label{t:equiv}
    Let $G$ be a residually (amenable, locally indicable, residually finite) group and let $Y$ be a free cocompact $G$-complex.
    Then
    \[
    b_{*}(Y; D_{\F G}) = \inf_{\substack{H <G \\
    [G:H] < \infty}} \frac{b_{*}(Y/H; \F)}{[G:H]}.
    \]
\end{Theorem}
\begin{proof}
    Let $H < G$ be a finite index subgroup of $G$.
    Note that the augmentation map $\F H \to \F$ exhibits $\F$ as an $\F H$-skew field, and hence the universality of $D_{\F H}$ implies
    \[
    b_{*}(Y; D_{\F H}) \leq b_{*}(Y; \F)= b_{*}(Y/H ; \F).
    \]
    The multiplicativity of the skew-field Betti numbers implies $\leq$ inequality in the desired formula.

    By considering the action of the stabilizers on connected components of $Y$ the computation of $b_{*}(Y; D_{\F G})$ always reduces to the finitely generated case, see \cite{aos24}*{Section 3} for details.
    So, for the opposite inequality it is enough to consider finitely generated $G$, Since skew-field Betti numbers are integers, applying \cite{j21}*{Theorem 1.2} to a residual sequence of amenable locally indicable residually finite quotients, implies that for some such quotient $K \nsgp G \onto A$ we have $b_{*}(Y; D_{\F G}) = b_{*}(Y/K; D_{\F A} )$.
    Applying \cite{lls11}*{Theorem 0.2} to the $A$-complex $Y/K$ allows us to find a further finite quotient such that $b_{*}(Y/K; D_{\F A})$ is approximated within any given $\epsilon$ by its normalized usual Betti numbers.
\end{proof}
In particular, if $X$ is a finite connected complex with residually (amenable locally indicable residually finite) fundamental group $G$, then $b_{*}(\widetilde X; D_{\F G }) = \lb(X; \F)$ as $b_{*}(\widetilde X; D_{\F G })$ is multiplicative by Lemma \ref{l:DG}(4).
Using Corollary~\ref{c:dom} in place of Fisher's theorem in the proof of Theorem B of \cite{aos24} allows us to drop the asphericity assumption.
\begin{Theorem}\label{t:B}
    If $X$ is a finite connected complex whose fundamental group is virtually RFRS, then $\lb_{\leq n}(X;\F )=0$ if and only if $\ub_{\leq n}(X;\F )=0$.
\end{Theorem}
\begin{proof}
    Since $\lb_k$ and $\ub_{k}$ are multiplicative we can assume that $\pi_1(X)$ is a finitely generated RFRS group.
    Hence, by Theorem~\ref{t:rvali}, $\pi_1(X)$ is residually (virtually abelian locally indicable), and we can apply Theorem~\ref{t:equiv} to get $\lb_{k}(X;\F )= b_k(X; D_{\F\pi_{1}X})$.

    So, suppose that $b_k(X; D_{\F \pi_1X})=0$ for all $k\leq n$.
    By Corollary~\ref{c:dom}, there is a finite cover $X'\to X$ and a map $\psi: \pi_1(X') \rightarrow \zz$ so that $C_\ast(\widetilde X; \F)$ is type $FP_n(\F \ker \psi)$.
    If $\hat X'$ is the cover of $X'$ corresponding to $\ker \psi$, then $X'$ is homotopy equivalent to the mapping torus $T_g$ of the covering translation $g:\hat X' \to \hat X'$, so $\ub_k(X'; \F )=\ub_k(T_g;\F )=0$ by Theorem~\ref{t:mappingtorustheorem}, and by multiplicativity of $\ub$ we conclude that $\ub_k(X;\F )=0$.
\end{proof}

\end{document}